\documentclass[12pt]{article}
\title{The local limit of unicellular maps in high genus}
\author{Omer Angel
  \and Guillaume Chapuy
  \and Nicolas Curien
  \and Gourab Ray}
\date{\small \today}

\usepackage{amsmath,amssymb,amsfonts,amsthm}

\usepackage[colorlinks=true,linkcolor=black,citecolor=black,urlcolor=blue,pdfborder={0 0 0}]{hyperref}

\usepackage{graphicx}
\usepackage{pst-all}

\usepackage[font=sf, labelfont={sf,bf}, margin=1cm]{caption}

\usepackage{cleveref}
  \crefname{theorem}{Theorem}{Theorems}
  \crefname{thm}{Theorem}{Theorems}
  \crefname{lemma}{Lemma}{Lemmas}
  \crefname{lem}{Lemma}{Lemmas}
  \crefname{remark}{Remark}{Remarks}
  \crefname{prop}{Proposition}{Propositions}
  \crefname{defn}{Definition}{Definitions}
  \crefname{corollary}{Corollary}{Corollaries}
  \crefname{section}{Section}{Sections}
  \crefname{figure}{Figure}{Figures}

\newtheorem{thm}{Theorem}[]
\newtheorem{open}{Question}[]
\newtheorem{lemma}[thm]{Lemma}

\newtheorem{prop}[thm]{Proposition}
\theoremstyle{definition}

\newtheorem{remark}[thm]{Remark}

\numberwithin{equation}{section}

\renewcommand{\P}{\mathbb P}

\newcommand{\E}{\mathbb E}

\newcommand{\N}{\mathbb N}

\newcommand{\cU}{\mathcal U}
\newcommand{\Tinf}{T^\infty}
\newcommand{\Geom}{\operatorname{Geom}}

\begin{document}
\maketitle

\begin{abstract}
  We show that the local limit of unicellular maps whose genus is
  proportional to the number of edges is a supercritical geometric
  Galton-Watson tree conditioned to survive. The proof relies on
  enumeration results obtained \emph{via} the recent bijection given by the
  second author together with F\'eray and Fusy.
\end{abstract}

\section{Introduction}\label{sec:intro}
Recently, the last author of this note studied the large scale structure of
random unicellular maps whose genus grows linearly with their
size~\cite{R13}. Our goal here is to identify explicitly the local limit of
the latter as a super-critical geometric Galton-Watson tree conditioned to
survive. \medskip

Motivated by the theory of two-dimensional quantum gravity, the study of
local limits (also known as Benjamini-Schramm limits \cite{BS01}) of
random planar maps and graphs has been rapidly developing over the last
years, since the introduction of the Uniform Infinite Planar Triangulation
(UIPT) by Angel \& Schramm~\cite{AS03}.  The UIPT is defined as the local
limit in distribution (see definition below) of uniform random
triangulations of the sphere, when their size tends to infinity.

It is natural to expect (see \cite{Gos12}) that, for any fixed $g\geq1$,
the UIPT is also the local limit of uniform random triangulations of a
surface of genus $g$ when their size tends to infinity (note that the
situation is totally different for \emph{scaling} limits, where the genus
affects the topology of the limiting surface~\cite{Bet}).  However, one
expects to obtain a totally different picture if one lets the genus of the
maps grow linearly with their size.  In this case, the limiting average
degree is strictly greater than in the planar case, so that some kind of
``hyperbolic'' behavior is expected, see \cite{AR13,Ben10,R13}. In this
note, we take a step in the study of this hyperbolic regime, by studying
the local limit of \emph{unicellular} maps whose genus is proportional to
their size.

Recall that a \emph{map} is a proper embedding of a finite connected graph
into a compact orientable surface considered up to oriented homeomorphisms,
and such that the connected components of the complement of the embedding
(called \emph{faces}) are topological disks.  Loops and multiple edges are
allowed, i.e.  our {graphs} are actually multigraphs.  As usual, all the
maps considered here are rooted, that is given with a distinguished
oriented edge.

Alternatively, a (rooted) map can be seen as a (rooted) graph together with
a cyclic orientation of the edges around each vertex.  This allows us to
view any connected subgraph of a map as a map structure, obtained by
restriction of the cyclic order.  (This can also be done in terms of the
embedding, but the surface must be modified to make all faces topological
discs.)  In particular, we can define the ball $B_{r}(m)$ to be the rooted
map obtained from $m$ by keeping all the edges and vertices which are at
distance less than or equal to $r$ from the origin of the root edge of $m$.
One can then define the \emph{local topology} \cite{AS03,BS01} between two
maps $m,m'$ (of arbitrary genera) using the metric
\[
d_{\mathrm{loc}}(m,m') = e^{\displaystyle -\sup \{r : B_r(m) \approx B_r(m')\}},
\]
where we write $M\approx M'$ if $M$ is isomorphic to $M'$ as maps.

A {\bf unicellular map} (or: {\bf one-face map}) is a map with only one
face.  This class attracted much attention, both because of its remarkable
enumerative and combinatorial properties (see, e.g.\ \cite{CFF} and
references therein), and because unicellular maps are the fundamental
building blocks in the study of general maps on surfaces and their scaling
limits (see, e.g.\ \cite{CMS, Bet}). In the planar case $g=0$, unicellular
maps are nothing more than trees.  For $n \geq 1$ and $g \geq 0$ denote by
$\cU_{g,n}$ the set of all unicellular maps with $n$ edges and genus $g$.
An application of Euler's characteristic formula shows that $v = n+1-2g$,
where $v$ is the number of vertices of the map.  In particular $\cU_{g,n} =
\varnothing$ as soon as $2g>n$.  For $g\leq n/2$ we shall denote by
$U_{g,n}$ a random map, uniformly distributed over $\cU_{g,n}$.

We write $\Geom(\xi)$ to denote a random variable which follows the
geometric distribution with parameter $\xi\in(0,1)$.  In other words,
\[
\P(\Geom(\xi) = k ) = (1-\xi)^{k-1}\xi \qquad \text{ for } k\ge 1.
\]
For any $\xi\in(0,1)$ we shall use $T_\xi$ to denote the Galton-Watson tree
with offspring distribution $\Geom(\xi)-1$.  For $\xi<1/2$ this tree is
super-critical.  We denote by $\Tinf_\xi$ the tree $T_\xi$ conditioned to be
infinite.

\begin{thm}
  \label{thm:main}
  Assume $g_n$ is such that $g_n/n\to\theta$ with $\theta\in[0,1/2)$. Then
  we have the following convergence in distribution for the local topology:
  \[
  U_{g_{n},n} \xrightarrow[n\to\infty]{(d)} \Tinf_{\xi_{\theta}},
  \]
  where $\xi_{\theta} = \frac{1-\beta_\theta}{2}$, and $\beta_\theta$ is
  the unique solution in $\beta \in [0,1)$ of
  \begin{equation}
    \label{eq:beta}
    \frac12 \left(\frac{1}{\beta}-\beta\right) \log \frac{1+\beta}{1-\beta}
    = (1-2\theta).
  \end{equation}
\end{thm}

Note that the mean of the geometric offspring distribution in
\cref{thm:main} is given by $(1+\beta_\theta)/(1-\beta_\theta)>1$ and in
particular the Galton-Watson tree is supercritical.

In order to prove \cref{thm:main} we first determine the root degree
distribution of unicellular maps using the bijection of \cite{CFF}.  This
is done in \cref{sec:enumeration}, where we also obtain an asymptotic
formula for $\#\cU_{g,n}$.  This enables us to compute in
\cref{sec:surgery} the probability that the ball of radius $r$ around the
root in $U_{g_{n},n}$ is equal to any given tree.  In \cite{R13} it is
shown that the local limit of unicellular maps is supported on trees.
However, we do not rely on this result.  In \cref{sec:identification} we
show that the probabilities computed below are sufficient to characterize
the local limit of $U_{g,n}$.

\section{Enumeration and root degree distribution}
\label{sec:enumeration}

We begin be describing a bijection from \cite{CFF} between unicellular maps
and trees with some additional structure. A {\bf $C$-decorated tree} is a
plane tree together with a permutation on its vertices whose cycles all
have odd length, carrying an additional sign $\{\pm1\}$ associated with
each cycle. The \emph{underlying graph} of a $C$-decorated tree is the
graph obtained from the tree by identifying the vertices in each cycle of
the permutation to a single vertex. Hence if the tree has $n$ edges and the
permutation has $v$ cycles, the underlying graph has $n$ edges and $v$
vertices (recall that we allow loops and multiple edges). We also note that
at any vertex $v$ of the tree which is a fixed point of the permutation,
the cyclic order on the edges around $v$ in the tree and in the resulting
unicellular map are the same. This will be of use in our analysis of the
case $g=o(n)$.

\begin{thm}[\cite{CFF}]
  \label{thm:bij}
  Unicellular maps with $n$ edges and genus $g$ are in $2^{n+1}$ to $1$
  correspondence with $C$-decorated trees with $n$ edges and $s=n+1-2g$
  cycles.  This correspondence preserves the underlying graph.
\end{thm}

Using this correspondence we will obtain the two main theorems of this
section, \cref{thm:root_deg,thm:enum}.  Before stating these theorems we
introduce a probability distribution on the odd integers that will play an
important role in the sequel.  For $\beta \in(0,1)$, we let $X_\beta$ be a
random variable taking its values in the odd integers, whose law is given
by:
\[
\P(X_\beta = 2k+1) := \frac{1}{Z_\beta} \frac{\beta^{2k+1}}{2k+1},
\]
where
\[
Z_\beta = \sum_{k\geq0} \frac{\beta^{2k+1}}{2k+1} =
  \frac12\log\frac{1+\beta}{1-\beta} = \operatorname{arctanh} \beta.
\]
It is easy to check that \cref{eq:beta} is equivalent to 
 \begin{equation}
   \label{eq:beta2}
   \E[X_{\beta}] = \frac{1}{Z_{\beta}} \frac{\beta}{1-\beta^2}
   = \frac{1}{1-2\theta}.
 \end{equation}

\begin{thm}\label{thm:enum}
  Assume $g_n\sim\theta n$. Let $\beta_n$ be such that $\E[ X_{\beta_n}] =
  \frac{n}{s} + o\left(n^{-1/2}\right)$ and $s_n=n+1-2g_n$.
  As $n$ tends to infinity we have
  \[
  \#\cU_{g_n,n} \sim A_\theta \, \frac{(2n)!}{n!s_n!\sqrt{s_n}} \,
  \frac{(Z_{\beta_n})^{s_n}}{4^{g_n} {\beta_n^{n+1}}}, 
  \]
  where
  $A_{\theta} = \frac{2}{\sqrt{2\pi \operatorname{Var}(X_{\beta_{\theta}})}}$.
\end{thm}

Note that $\beta_n\to\beta_\theta$. If $g=\theta n + o(\sqrt{n})$ we may
take $\beta_n$ to be just $\beta_\theta$ and not depend on $n$.

\begin{proof}
  For $s,n\geq 1$, let $\mathcal{L}_s(n+1)$ be the set of partitions of
  $n+1$ having $s$ parts, all of odd size.  Recall that if $\lambda$ is a
  partition of $n+1$, the number of permutations having cycle-type
  $\lambda$ is given by
  \[
  \frac{(n+1)!}{\prod_i m_i! i^{m_i}},
  \]
  where for $i\geq 1$, $m_i=m_i(\lambda)$ is the number of parts of
  $\lambda$ equal to $i$.  Therefore by \cref{thm:bij}, the number of
  unicellular maps of genus $g_n$ with $n$ edges is given by
  \begin{equation}\label{eq:nbCTrees}
    \#\cU_{g_n,n} =
    \mathrm{Cat}(n) \frac{2^{s_n}}{2^{n+1}}
    \sum_{\lambda \in \mathcal{L}_{s_n}(n+1)} \frac{(n+1)!}{\prod_i m_i!
      i^{m_i}},
  \end{equation}
  where $\mathrm{Cat}(n)=\frac{(2n)!}{n!(n+1)!}$ is the $n$th Catalan number,
  i.e.\ the number of rooted plane trees with $n$ edges, the sum counts
  permutations, and the powers of $2$ are from the signs on cycles of the
  permutation and since the correspondence is $2^{n+1}$ to $1$.
  This is known as the Lehman-Walsh formula (\cite{LW}).

  Now, let $\beta\in(0,1)$ and let $X_1,X_2,\dots,X_s$ be i.i.d.\ copies of
  $X_\beta$.  By the local central limit theorem \cite[Chap.7]{Petrov}, if
  $n+1 = s\E[ X_\beta] + o(\sqrt{s})$ has the same parity as $s$, then
  $\P(\sum_{i\le s} X_i = n+1) \sim A s^{-1/2}$ where $A = 2/ \sqrt{2 \pi
    \mathrm{Var}(X_{\beta})}$.  The additional factor of $2$ comes from the
  parity constraint since $X_i$ is odd.  On the other hand, we have
  \begin{align*}
    \P\left(\sum_{i\le s} X_i = n+1\right)
    &= \sum_{\stackrel{k_1+\dots+k_s = n+1}{k_i\text{ odd}}} \prod_i
    \frac{\beta^{k_i}}{Z_\beta\cdot k_i} \\ 
    &= \frac{\beta^{n+1}}{(Z_\beta)^s} \sum_{\lambda \in
      \mathcal{L}_{s}(n+1)} \frac{s!}{\prod_i m_i! i^{m_i}},
  \end{align*}
  since $\frac{s!}{\prod_{i}m_i!}$ is the number of distinct ways to order
  of the parts of the partition $\lambda$. 

  Therefore if, as in the statement of the theorem, we pick $\beta_{n}$ so
  that $ \mathbb{E}[X_{\beta_{n}}] = (n+1)/s + o( 1/\sqrt{n})$, noticing
  that $\beta_{n} \to \beta_{\theta}$ and $ \mathrm{Var}(X_{\beta_{n}}) \to
  \mathrm{Var}(X_{\beta_{\theta}})$, it follows from \cref{eq:nbCTrees} and
  the last considerations that
  \[
  \#\cU_{g_n,n} \sim
  \frac{1}{2^{2g_n}}\mathrm{Cat}(n)\frac{(n+1)!}{s!}
  \frac{(Z_{\beta_{n}})^s}{\beta^{n+1}_{n}} A_{\theta} s^{-1/2}. 
  \qedhere
  \]
\end{proof}

The following theorem gives an asymptotic enumeration of unicellular maps
of high genus with a prescribed root degree.

\begin{thm}[Root degree distribution]\label{thm:root_deg}
  Assume $g_n \sim \theta n$ with $\theta\in(0,1/2)$, and let
  $\beta_\theta$ be the solution of \cref{eq:beta}. Then for every $d \in
  \N$ we have
  \[
  \P\left( U_{g_{n},n} \mbox{ has root degree }d\right)
  \xrightarrow[n\to\infty]{}
  \left(\frac{1-\beta_\theta^2}{4}\right) \frac{(1+\beta_\theta)^d -
    (1-\beta_\theta)^d}{2^d\beta_\theta}.
  \]
  Equivalently, the degree of the root of $U_{g_n,n}$ converges in
  distribution to an independent sum $\Geom(\frac{1+\beta_\theta}{2}) +
  \Geom(\frac{1-\beta_\theta}{2}) - 1$.
\end{thm}

\begin{proof}

  As in the proof of \cref{thm:enum}, we see that the length of a uniformly
  chosen cycle in a uniform random $C$-decorated tree with $n$ edges and
  $n+1-2g_n$ cycles is distributed as the random variable $X_1$ conditioned
  on the fact that $X_1+\dots+X_s=n+1$, where the $X_i$'s are i.i.d.\
  copies of $X_\beta$ for any choice of $\beta \in (0,1)$, and
  $s=n+1-2g_n$.  Using the local central limit theorem, we see that with
  $\beta_n$ chosen according to \cref{thm:enum}, when $n$ tends to
  infinity, this random variable converges in distribution to
  $X_{\beta_\theta}$.

  Since the permutation is independent of the tree, the probability that a
  cycle contains the root vertex is proportional to its size. Therefore the
  size of the cycle containing the root vertex converges in distribution to
  a size-biased version of $X_{\beta_\theta}$, which is a random variable
  $K$ with distribution $\P(K=2k+1) = (1-\beta_\theta^2)\beta_\theta^{2k}$,
  i.e.\ $K=2\Geom(1-\beta_\theta^2)-1$.

  Now by \cref{thm:bij}, conditionally on the fact that the cycle
  containing the root vertex has length $2k+1$, the root degree in
  $U_{g_{n},n}$ is distributed as $\sum_{i=0}^{2k} D_i$, where $D_0$ if the
  degree of the root of a random plane tree of size $n$, and $(D_i)_{i>0}$
  are the degrees of $2k$ uniformly chosen distinct vertices the tree.  It
  is classical, and easy to see, that when $n$ tends to infinity the
  variables $(D_i)_{i>0}$ converge in distribution to independent
  $\Geom(1/2)$ random variables, while $D_0$ converges to $Y+Y'-1$, where
  $Y,Y'$ are further independent $\Geom(1/2)$ variables.  All geometric
  variables here are also independent of $K$.

  From this it is easy to deduce that when $n$ tends to infinity, the root
  degree in $U_{g_n,n}$ converges in law to $\sum_{i=0}^{K} Y_i -1$ where
  $K$ is as above and the $Y_i$'s are independent $\Geom(1/2)$ variables.
  Since the probability that the sum of $\ell$ i.i.d.\ $\Geom(1/2)$ random
  variables equals $m$ is $2^{-m} \binom{m-1}{\ell-1}$, we thus obtain that
  for all $d\geq 1$, the probability that the root vertex has degree $d$
  tends to:
  \[
  \frac{1-\beta_\theta^2}{\beta_\theta} \sum_{k\geq 0} \beta_\theta^{2k+1}
  2^{-d-1} \binom{d}{2k+1} 
  =
  \frac{1-\beta_\theta^2}{4\beta_\theta}
  \frac{(1+\beta_\theta)^{d}-(1-\beta_\theta)^{d}}{2^d}.
  \qedhere
  \]
\end{proof}

\begin{remark}
It may be possible to prove Theorem~\ref{thm:root_deg} using the enumeration
results for unicellular
maps by vertex degrees found in \cite{GoupilSchaeffer}, although this would
require some computations. Here we prefer to prove it using the 
bijection of \cite{CFF}, since the proof is quite direct and gives a good understanding of
the probability distribution that arises. This is also the reason we 
prove \cref{thm:enum} from the bijection, rather than starting directly
from the Lehman-Walsh formula \eqref{eq:nbCTrees}.
\end{remark}

We now comment on a ``paradox'' that the reader may have noticed. For any
rooted graph $G$ and any $r \geq0$ we denote by $V_{r}(G)$ the set of
vertices which are at distance less than $r$ from the origin of the graph.
In $U_{g,n}$ the mean degree can be computed as
$$ \lim_{r \to \infty}  \frac{1}{\#V_{r}(U_{g,n})} \sum_{u \in
  V_{r}(U_{g,n})} \mathrm{deg}(u) = \frac{2n}{v}
\xrightarrow[n\to\infty]{}	2(1-2\theta)^{-1}.$$
However, if one interchanges  $\lim_{n\to \infty}$ and $\lim_{r \to
  \infty}$ a different larger result appears. Indeed, easy arguments about
Galton-Watson processes show that in $T_{\xi_{\theta}}^\infty$ we have
$$ \lim_{r \to \infty} \frac{1}{\# V_{r}(T_{\xi_{\theta}}^\infty)}\sum_{u
  \in V_{r}(T_{\xi_{\theta}}^\infty)} \mathrm{deg}(u)=
\frac{2}{1-\beta_{\theta}}.$$

%

\subsection{The low genus case}

\begin{proof}[Proof of \cref{thm:main} for $\theta=0$]
  As noted, the case $g=0$ is well known. We argue here that the local
  limit for $g=o(n)$ is the same as for $g=0$. Indeed, the permutation on
  the tree contains $n+1-2g$ cycles, and so has at most $2g$ non-fixed
  points. (If cycles of length $2$ were allowed this would be $4g$.) Since
  the permutation is independent of the tree, and since the ball of radius
  $r$ in the tree distance is tight, the probability that any vertex in the
  ball is in a non-trivial cycle is $o(1)$ (with constant depending on
  $r$). In particular, the local limit of the unicellular map and of the
  tree are the same.
\end{proof}

\section{The local limit}

\subsection{Surgery} 
\label{sec:surgery}

Throughout this subsection, we fix integers $n,g \geq 0$. Let $t$ be a
rooted plane tree of height $r \geq 1$ with $k$ edges and exactly $d$
vertices at height $r$.

\begin{lemma} \label{lem:surgery}
  For any $n,g,k,d,r \geq 0$ we have 
  \[
  \# \big\{ m \in \cU_{g,n} : B_{r}(m) = t \big\}  =
  \# \big\{m \in \cU_{g,n-k+d} \mbox{ with root degree } d \big\}.
  \]
\end{lemma}

\begin{proof}
  The lemma follows from a surgical argument illustrated in
  Fig.~\ref{fig:surgery}: if $m \in \mathcal{U}_{g,n}$ is such that $
  B_{r}(m) = t$ we can replace the $r$-neighborhood of the root by a star
  made of $d$ edges which dimishes the number of edges of the map by $k-d$
  and turn it into a map of $ \mathcal{U}_{g,n-k+d}$ having root degree
  $d$. To be precise, consider the leaf of $t$ first reached in the contour
  around $t$. The edge to this leaf is taken to be the root of the new map.

  \begin{figure}[!ht]
    \begin{center}
      \includegraphics[width=11cm]{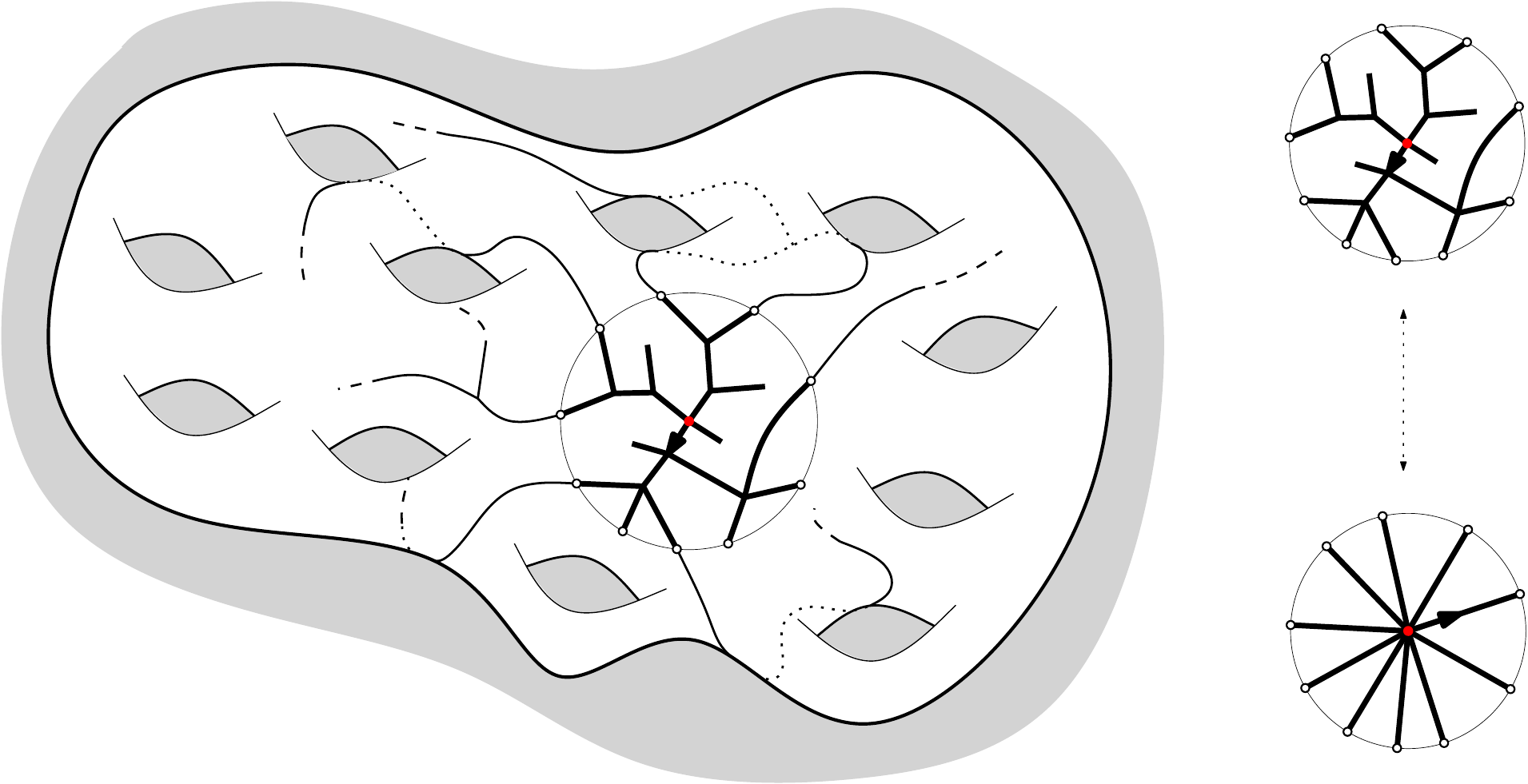}
      \caption{Illustration of the surgical operation}
      \label{fig:surgery}
    \end{center}
  \end{figure}

  It is clear that this operation is invertible. To see that it is a
  bijection between the two sets in question we need to establish that it
  does not change the genus or number of faces in a map. One way to see
  this is based on an alternative description of the surgery, namely that
  it contracts every edge of $t$ except those incident to the leaves, and
  it is easy to see that edge contraction does not change the number of
  faces or genus of a map.

\end{proof}

\subsection{Identifying the limit} \label{sec:identification}

Recall that for $\xi\in(0,1)$ we denote by $T_\xi$ the law of a
Galton-Watson tree with $\Geom(\xi)-1$ offspring distribution.  Note that
when $\xi \in (0,1/2)$ the mean offspring is strictly greater than $1$ and
so the process is supercritical, and recall that $\Tinf_\xi$ is $T_\xi$
conditioned to survive.  Plane trees can be viewed as maps, rooted at the
edge from the root to its first child.  For every $r \geq 0$, if $t$ is a
(possibly infinite) plane tree we denote by $B_r(t)$ the rooted subtree of
$t$ made of all the vertices at height less than or equal to $r$.

\begin{prop}\label{prop:GW}
  Fix $\xi \in (0,1/2)$.  For any tree $t$ of height exactly $r$ having $k$
  edges and exactly $d$ vertices at maximal height, we have
  \[
  \P\big(B_r(\Tinf_\xi) = t \big)
  = \frac{\big(\xi (1-\xi)\big)^{k+1-d} \big((1-\xi)^d - \xi^d\big)}
         {1-2\xi}.
  \]
\end{prop}

Note that the probability of observing $t$ does not depend on $r$, but only
on the number of edges and vertices where $t$ is connected to the rest of
$T_\xi$.

\begin{proof}
  Since $\xi\in(0,1/2)$ the Galton-Watson process is supercritical and
  by standard result the extinction probability $p_{\mathrm{die}}$ is
  strictly less than $1$ and is the root of $x = \sum_{k\ge0} x^k (1-\xi)^k\xi$
  in $(0,1)$.  Hence 
  \[
  p_{\mathrm{die}} = \frac{\xi}{1-\xi}.
  \]
  Next, fix a tree $t$ of height exactly $r$ with $k$ edges and $d$
  vertices at height $r$.  By the definition of $T_\xi$ if $k_u$ denotes
  the number of children of the vertex $u$ in $t$ we have
  \[
  \P(B_r(T_\xi) = t) = \prod_u (1-\xi)^{k_u}\xi = (1-\xi)^k \xi^{k+1-d} 
  \]
  where the product is taken over all the vertices of $t$ which are at
  height less than $r$.  Conditioned on the event $\{B_r(T_\xi) = t\}$, by
  the branching property, the probability that the tree survives forever is
  $(1 - p_{\mathrm{die}}^d)$.  Combining the pieces, we get the statement
  of the proposition.
\end{proof}

\begin{proof}[Proof of \cref{thm:main} for $\theta \in (0,1/2)$]
  Under the assumptions of \cref{thm:main}, fix $r$ and let $t$ be a rooted
  oriented tree of height exactly $r$ having $k$ edges and exactly $d$
  vertices at height $r$.  By \cref{lem:surgery} we have
  \begin{align*}
    \P(B_r(U_{n,g_n}) = t)
    &= \frac{\# \{m \in \cU_{g_n,n-k+d} \mbox{ with root degree } d \}}
    {\# \cU_{g_n,n}} \\ 
    &= \frac{\# \cU_{g_n,n-k+d}}{\# \cU_{g_n,n}} \cdot \P(\text{root degree
      of } U_{g_n,n-k+d} = d).
  \end{align*}
  Applying \cref{thm:root_deg} we have 
  \begin{equation}
    \label{eq:un}
    \P(\text{root degree of }U_{g_n,n-k+d} = d) \xrightarrow[n\to\infty]{}
    \left(\frac{1-\beta_\theta^2}{4\beta_\theta}\right)
    \frac{(1+\beta_\theta)^d - (1-\beta_\theta)^d}{2^d}.
  \end{equation}
  On the other hand, since $n/s = (n-k+d)/(s-k+d) + o(1/ \sqrt{n})$ we can
  apply \cref{thm:enum} for the asymptotic of $\#\cU_{g_n,n-k+d}$ and
  $\#\cU_{g_n,n}$ \emph{with the same sequence} $(\beta_n)$ and get that
  \[
  \frac{\# \cU_{g_n,n-k+d}}{\# \cU_{g_n,n}} \sim
  \frac{(2n+2d-2k)! n! s! Z_{\beta_{n}}^{d-k}}{(2n)! (n+d-k)! (s+d-k)!
    \beta_{n}^{d-k}}.
  \]
  Since $d,k$ are fixed, and using the facts that $\beta_n \to
  \beta_\theta$, $Z_{\beta_n} \to Z_{\beta_\theta}$ and $s/n\to
  (1-2\theta)$, the last display is also equivalent to   
  \begin{equation}
    \frac{\# \cU_{g_n,n-k+d}}{\# \cU_{g_n,n}} \sim
    \left(\frac{\beta_\theta(1-2\theta)}{4Z_{\beta_\theta}} \right)^{k-d}
    = \left( \frac{1-\beta_\theta^2}{4} \right)^{k-d}, \label{eq:deux}
  \end{equation}
  by the definition of $\beta_\theta$ in \cref{eq:beta2}.  Plugging
  \eqref{eq:un} and \eqref{eq:deux} together and using \cref{prop:GW} we
  find that
  \[
  \P(B_r(U_{g_n,n}) = t) \xrightarrow[n\to\infty]{}
  \P(B_r(\Tinf_{\xi_\theta}) =t),
  \]
  with $\xi_\theta = (1-\beta_\theta)/2$.

  Finally, note that the law of $B_r(\Tinf_{\xi_\theta})$ is a probability
  measure on the set of finite plane trees.  It follows that
  $B_r(U_{g_n,n})$ is tight, and converges in distribution to
  $B_r(\Tinf_{\xi_{\theta}})$. Since $r$ is arbitrary, this completes
  the proof of the Theorem.
\end{proof}

\section{Questions and remarks}

\paragraph{Planarity.}
A consequence of \cref{thm:main} is that $U_{g_{n},n}$ is locally a tree
(hence planar) near its root. More precisely, the length of a minimal
non-trivial cycle containing the root edge diverges in probability as $n
\to \infty$. A much stronger statement has been proved in \cite{R13} where
quantitative estimates on cycle lengths are obtained. As noted above, our
proof does not rely on this result and our approach is softer. Note that
our method of proof only requires to prove convergences of the quantities
$\P(B_{r}(U_{g_{n},n}) = t)$ when $t$ is a tree since we were able to
identify these limits as coming from a probability measure on infinite
trees.

\paragraph{Open questions.}
We gather here a couple of possible extensions of our work.

\begin{open}
  Find more precise asymptotic formulae for $\#\cU_{g,n}$ as
  $g,n\to\infty$. \cref{thm:enum} gives a first order approximation.
\end{open}

\begin{open}
  Quantitatify the convergence of $U_{g_{n},n}$ to $T_{\xi_{\theta}}$. In
  particular, let $r_n = o(\log n)$. Is it possible to couple $U_{g_n,n}$
  with $T_{\xi_\theta}$ so that $B_{r_n}(U_{g_n,n}) =
  B_{r_n}(T_{\xi_\theta})$ with high probability?
\end{open}


\bibliographystyle{siam}
\bibliography{biblio}

\noindent
{\sc Omer Angel, Gourab Ray} \\
Department of Mathematics, University of British Columbia, Canada\\
{\em \{angel,gourab\}@math.ubc.ca}\\

\noindent
{\sc Guillaume Chapuy}\\
CNRS and LIAFA Universit\'e Paris Diderot - Paris 7, France\\
{\em guillaume.chapuy@liafa.univ-paris-diderot.fr} \\

\noindent
{\sc Nicolas Curien}\\
CNRS and LPMA Universit\'e Pierre et Marie Curie - Paris 6, France\\
{\em nicolas.curien@gmail.com}\\

\end{document}